\documentclass{amsart}
\usepackage{amsmath}
\usepackage{amsfonts}
\usepackage{amssymb}
\usepackage{graphicx}
\usepackage{amsthm}
\usepackage{newlfont}

 \newtheorem{thm}{Theorem}[section]
 \newtheorem{cor}[thm]{Corollary}
 \newtheorem{lem}[thm]{Lemma}
 \newtheorem{prop}[thm]{Proposition}
 \newtheorem{defn}[thm]{Definition}
\theoremstyle{remark}

 \numberwithin{equation}{section}

\begin{document}

\title[Subgroup S--commutativity degree of finite groups]
 {Subgroup S--commutativity degree of finite groups}


\author[D. E. Otera]{Daniele Ettore Otera}
\address{D\'epartement de Math\'ematique,Universit\'e Paris-Sud 11\\
Batiment 425, Facult\'e de Science d'Orsay\\
F-91405, Orsay Cedex, France} \email{daniele.otera@math-psud.fr}

\author[F. G. Russo]{Francesco G. Russo}
\address{
Department of Mathematics, University of Palermo\\
Via Archirafi 34, 90123, Palermo, Italy.
} \email{francescog.russo@yahoo.com}




\date{\today}

\begin{abstract}
The so--called subgroup commutativity degree $sd(G)$ of a finite group $G$ is the number of permuting subgroups $(H,K) \in \mathrm{L}(G) \times \mathrm{L}(G)$, where $\mathrm{L}(G)$ is the subgroup lattice of $G$, divided by $|\mathrm{L}(G)|^2$. It allows us to measure how $G$ is far from the celebrated classification of quasihamiltonian groups of K. Iwasawa. Here we generalize $sd(G)$, looking at suitable sublattices of $\mathrm{L}(G)$, and show some new lower bounds.  
\end{abstract} 

\maketitle

\section{Introduction}
All groups in the present paper are supposed to be finite. Given two subgroups $H$ and $K$ of a
group $G$, the product $HK = \{hk \ | \ h \in H, k \in K\}$ is not always a subgroup of $G$. $H$ and $K$
\textit{permute}  if $HK=KH$, or equivalently, if $HK$ is a subgroup of $G$. $H$ is said to be
\textit{permutable} (or \textit{quasinormal}) in $G$, if it permutes with every subgroup of $G$. It is possible
to strengthen this notion in various ways.   $H$  is
\textit{S--permutable} (or \emph{S--quasinormal}) in $G$, if $H$ permutes with all Sylow  subgroups of $G$
(for all primes in the set $\pi(G)$ of the prime divisors of $|G|$). Historically, O. Kegel introduced S--permutable subgroups in 1962, generalizing a well--known result of O. Ore of 1939, who proved that permutable subgroups
are subnormal (see \cite{ls,rs} for details). Roughly speaking, this notion deals with subgroups which are permutable with maximal subgroups. Several authors investigated the topic in the successive years and we mention \cite{a,bh,ps,rs}  for our aims.

The \textit{subgroup lattice} $\mathrm{L}(G)$ of a group $G$  is the set of all subgroups of $G$ and is a complete bounded lattice with respect to the set inclusion, having initial element the
trivial subgroup $\{1\}$ and final element $G$ itself (see \cite{gg,rs}). Its binary operations $\wedge, \vee$ are
defined by $X \wedge Y=X\cap Y$, $X\vee Y=\langle X \cup Y\rangle$, for all $X,Y\in
\mathrm{L}(G)$. Furthermore,  $\mathrm{L}(G)$ is \textit{modular}, if all the subgroups of $G$ satisfy the
\textit{modular law}. $G$ is  modular, if $\mathrm{L}(G)$ is modular (see \cite[Section
2.1]{rs}). This notion is important, because of the following concept. A group $G$ is \textit{quasihamiltonian},
if all its subgroups are permutable. By a result of K. Iwasawa  \cite[Theorem 2.4.14]{rs}, quasihamiltonian
groups are classified, but, at the same time, these groups are
characterized to be nilpotent and modular (see \cite[Exercise 3, p.87]{rs}).


Now we recall some terminology from \cite{ps}, which will be useful in the rest of the paper. Any non--empty set of subgroups $\mathrm{S}(G)$ of $G$ may be always regarded  as a sublattice of $\mathrm{L}(G)$ having initial element ${\underset{S \in \mathrm{S}(G)}\bigwedge} S$ and final element  ${\underset{S \in \mathrm{S}(G)}\bigvee} S$. The symbol $\mathrm{S}^\perp(G)$ denotes the set of all subgroups $H$ of $G$  which are permutable with all $S \in \mathrm{S}(G)$ and it is easy to check that
$\mathrm{S}^\perp(G)$ is a sublattice  of  $\mathrm{L}(G)$ (see \cite[Section 1]{ps}). There is a wide literature when we choose
$\mathrm{S}(G)$ to be equal to the sublattice $\mathrm{M}(G)$ of all maximal subgroups of $G$, or to
the sublattice $\mathrm{sn}(G)$ of all subnormal subgroups of $G$, or also to
the sublattice $\mathrm{n}(G)$ of all normal subgroups of $G$. Consequently, 
$\mathrm{L}^\perp(G)$ is the sublattice of all permutable subgroups of $G$, 
$\mathrm{M}^\perp(G)$ that of the subgroups permutable with all maximal subgroups of $G$ and so on for
$\mathrm{sn}^\perp(G)$ and $\mathrm{n}^\perp(G)=\mathrm{L}(G)$. Immediately, the role of the operator $\perp$ appears to be very intriguing for the structure of $G$ and several authors investigated this aspect. For instance, $G$ is  quasihamiltonian if and only if $\mathrm{L}(G)=\mathrm{L}^\perp(G)$.

In Section 2 we will describe a notion of probability on $\mathrm{L}(G)$, beginning from groups  in which the subgroups in $\mathrm{sn}(G)$ permutes with those in $\mathrm{M}(G)$. The generality of the methods  (we follow \cite{cms,dn,elr,r1,mst,nr,r2,r3,mt}) may be translated in terms of arbitrary sublattices, satisfying a prescribed restriction.  Section 3 shows some consequences on the size of $|\mathrm{L}(G)|$.

\section{Measure theory on subgroup lattices}
The following notion has analogies with \cite[Definitions 2.1,3.1,4.1]{elr} and \cite[Equation 1.1]{r2} 
and will be treated as in \cite{cms,dn,elr,r1,mst,nr,r2,r3,mt}.

\begin{defn}\label{d:1}For a group $G$,
\begin{equation}\label{ed:1}spd(G)= \frac{|\{(X,Y)\in \mathrm{sn}(G) \times \mathrm{M}(G) \ | \ XY=YX\}|}{|\mathrm{sn}(G)| \ |\mathrm{M}(G)|},\end{equation}  is the subgroup S--commutativity  degree of $G$.
\end{defn}
$0<spd(G)\le 1$ denotes the  probability that a randomly picked pair $(X,Y) \in \mathrm{sn}(G) \times \mathrm{M}(G)$ is permuting, that is, $XY=YX$.  \eqref{ed:1} may be rewritten, introducing
the function $\chi :
\mathrm{sn}(G) \times \mathrm{M}(G) \rightarrow \{0,1\}$ defined by
\begin{equation}\chi(X,Y)=\left\{\begin{array}{lcl} 1,&\,\,& \mathrm{if} \ XY=YX,\\
0,&\,\,& \mathrm{if} \ XY\not= YX,\end{array}\right.\end{equation} in the following form
\begin{equation}\label{e:temp3}
spd(G)= \frac{1}{|\mathrm{sn}(G)| \ |\mathrm{M}(G)|} {\underset{(X,Y)\in \mathrm{sn}(G) \times \mathrm{M}(G)} \sum}\chi(X,Y).
\end{equation}
In Definition \ref{d:1} and \eqref{e:temp3}, we may replace $\mathrm{sn}(G) \times \mathrm{M}(G)$ with $\mathrm{S}(G) \times \mathrm{T}(G) $, where $\mathrm{S}(G)$ and $\mathrm{T}(G)$ are two arbitrary sublattices of $\mathrm{L}(G)$. We have chosen $\mathrm{sn}(G) \times \mathrm{M}(G)$,  because  \cite{a,bh} describe the structure of the groups in which the subnormal subgroups permute with all Sylow subgroups (called $PST$--\textit{groups}). If $\mathrm{Syl}(G)$ is the set of all Sylow subgroups of $G$, $\mathrm{Syl}(G)\subseteq \mathrm{M}(G)$ and this means that we have already a classification for a group $G$ such that $\mathrm{sn}(G) \subseteq \mathrm{Syl}(G)^\perp$.

\eqref{e:temp3} allows us to  to treat the problem from the point of view of the measure theory on groups. A computational advantage may be found in a formula for $spd(G_1 \times G_2)$, where $G_1$ and $G_2$ are two given groups. 
\begin{cor}\label{c:1} Let $G_i$ be a family of groups of coprime orders for $i=1,2,\ldots,k$. Then
$spd(G_1 \times G_2 \times \ldots \times G_k)=spd(G_1) \ spd(G_2) \ \ldots \ spd(G_k).$ 
\end{cor}
The techniques of proof are straightforward applications of \eqref{e:temp3} and the details are omitted. However, it is good to note that Corollary \ref{c:1} shows the stability  with respect to forming direct products of $spd(G)$: this  fact was proved in \cite{cms,dn,elr,r1,mst,r2,r3,mt} in different contexts. Another basic property is to relate $spd(G)$ to quotients  and subgroups of $G$.

Let $G=NH$ for a normal subgroup $N$ of $G$ and a subgroup $H$ of $G$ isomorphic to $G/N$ (briefly, $H \simeq G/N$). In general, it is easy to check that $\mathrm{sn}(G/N)$ is lattice isomorphic to $\mathrm{sn}(H)$ (briefly, $\mathrm{sn}(G/N) \sim \mathrm{sn}(H)$) and that $\mathrm{M}(G/N) \sim \mathrm{M}(H)$. 
We will concentrate  on some special classes of groups, satisfying 
\begin{equation}\label{a:1}
\Big(\mathrm{sn}(G/N) \times \mathrm{M}(G/N)\Big) \sim \Big(\mathrm{sn}(H) \times \mathrm{M}(H)\Big) \subseteq \mathrm{sn}(G) \times \mathrm{M}(G)\end{equation}
\begin{equation}\label{a:2}  \mathrm{sn}(N) \times \mathrm{M}(N) \subseteq \mathrm{sn}(G) \times \mathrm{M}(G).
\end{equation} 
\eqref{a:1}--\eqref{a:2}, jointly with \eqref{e:temp3}, allow us to conclude
\begin{equation}\label{lb1}   \sum_{(X,Y)\in \mathrm{sn}(G) \times \mathrm{M}(G)} \chi(X,Y) \ge  \sum_{(X,Y)\in \mathrm{sn}(N) \times \mathrm{M}(N)} \chi(X,Y);\end{equation}
{\small
\begin{equation}\label{lb2}  \sum_{(X,Y)\in \mathrm{sn}(G) \times \mathrm{M}(G)}\ge  {\underset{(X/N,Y/N)\in \mathrm{sn}(G/N) \times \mathrm{M}(G/N)} \sum} \chi(X/N,Y/N)=\sum_{(Z,T)\in \mathrm{sn}(H) \times \mathrm{M}(H)} \chi(Z,T)\end{equation}
}
and  consequently 
\begin{equation}\label{lb3} 2 |\mathrm{sn}(G)| \ |\mathrm{M}(G)| \ spd(G)\ge   \sum_{(X,Y)\in  \mathrm{sn}(N) \times \mathrm{M}(N)} \chi(X,Y) + \sum_{(Z,T)\in  \mathrm{sn}(H) \times \mathrm{M}(H)} \chi(Z,T).\end{equation}
\[=    |\mathrm{sn}(N)| \ |\mathrm{M}(N)| \ spd(N) + |\mathrm{sn}(G/N)| \ |\mathrm{M}(G/N)| \ spd(G/N).\]
Similar techniques have been used by  T\v{a}rn\v{a}uceanu \cite{mt} in order to study
the \textit{subgroup commutativity degree}
\begin{equation}sd(G)=\frac{|\{(X,Y)\in \mathrm{L}(G)^2 \ | \ XY=YX\}|}{|\mathrm{L}(G)|^2}=\frac{1}{|\mathrm{L}(G)|^2}{\underset{(X,Y)\in
\mathrm{L}(G)^2}\sum}\chi(X,Y).\end{equation}  \cite{mt} can be seen as a natural extension, to the context of the lattice theory, of the concept of \textit{commutativity degree}
\begin{equation}d(G)=\frac{|\{(x,y)\in G^2 \ | \ xy = yx\}|}{|G|^2}=\frac{1}{|G|^2}{\underset{x\in
G}\sum}|C_G(x)|,\end{equation} where $C_G(x)=\{g\in G \ | \ gx=xg\}$. There are several contributions on $d(G)$ in \cite{cms,dn,elr,r1,mst,nr,r2,r3}. The main strategy of investigation is to begin with the case of equality at 1 and then  describe the situation, when we leave this extremal case. Upper and lower bounds will measure the distance from known classes of groups. 
For instance, $d(G)=1$ if and only if $G$ is abelian; $sd(G)=1$ if and only if $\mathrm{L}(G)=\mathrm{L}(G)^\perp$. Therefore the next  are   mile stones for the rest of the paper.

\begin{cor}\label{c:2}In a group $G$ we have $spd(G)=1$ if and only if $\mathrm{sn}(G)\subseteq \mathrm{M}^\perp(G)$ or  $\mathrm{M}(G)\subseteq \mathrm{sn}^\perp(G)$.
\end{cor}

\begin{proof}It follows from the above considerations.
\end{proof}

\begin{cor}\label{c:3}If $G$ is a nilpotent group, then $spd(G)=1$.
\end{cor}

\begin{proof}Application of Corollary \ref{c:2}, noting that $\mathrm{M}(G) \subseteq \mathrm{n}(G) \subseteq \mathrm{sn}^\perp(G)$.
\end{proof}




\begin{cor}\label{c:4}In a group $G$ we have $\frac{|\mathrm{sn}(G)| \ |\mathrm{M}(G)|}{|\mathrm{L}(G)|^2} \ spd(G) \le \ sd(G)$ and the equality holds if and only if $\mathrm{sn}(G)=\mathrm{M}(G)=\mathrm{L}(G)$.
\end{cor}

\begin{proof}Since $\mathrm{sn}(G) \times \mathrm{M}(G) \subseteq \mathrm{L}(G)^2$,
$\{(X,Y)\in \mathrm{sn}(G) \times \mathrm{M}(G) \ | \ XY=YX\} \subseteq \{(X,Y)\in \mathrm{L}(G)^2 \ | \ XY=YX\}$ and then
\[
|\mathrm{sn}(G)| \ |\mathrm{M}(G)| \ spd(G)=|\{(X,Y)\in \mathrm{sn}(G) \times \mathrm{M}(G) \ | \ XY=YX\}|\]
\[ \leq |\{(X,Y)\in \mathrm{L}(G)^2 \ | \ XY=YX\}|= |\mathrm{L}(G)|^2 \ sd(G)
\]
from which the inequality follows. The rest is clear.
\end{proof}



Corollary \ref{c:3} clarifies the situation for nilpotent groups. Then we proceed to study solvable groups. Unfortunately, these cannot be described as in \cite[Proposition 2.4]{mt}: Different techniques are necessary. We recall that an abelian group $A$ of order $n=p_1^{n_1} p_2^{n_2} \ldots p_m^{n_m}$, for suitable powers of $p_1,p_2,\ldots,p_m\in \pi(A)$,  has a canonical decomposition of the form $A\simeq A_1 \times A_2 \times \ldots A_m,$ where $n_1, \ldots, n_m$ are positive integers and $A_1,A_2 \ldots, A_m$ are the primary factors. It is well--known that $|\mathrm{L}(A)|=|\mathrm{L}(A_1)|  \cdot |\mathrm{L}(A_2)| \cdot \ldots  \cdot |\mathrm{L}(A_m)|.$ In case $p=p_1=p_2= \ldots=p_m$  \cite[Proposition 3.2]{mtbis} shows that the number of maximal subgroups of the elementary abelian $p$--group $\mathbb{Z}_{p^{\alpha_1}} \times \mathbb{Z}_{p^{\alpha_2}}  \times \ldots \times \mathbb{Z}_{p^{\alpha_k}} $ is equal to $\frac{p^k-1}{p-1}$, for suitable integers $1 \le \alpha_1 \le \alpha_2 \le \ldots \le \alpha_k$ and $k\ge1$.




\begin{lem}\label{l:1} Let $N=\mathbb{Z}_{p^{\alpha_1}} \times \mathbb{Z}_{p^{\alpha_2}}$ be a non--trivial normal abelian subgroup of $G$ with $0\le \alpha_1+ \alpha_2$ and $1 \le \alpha_1 \le \alpha_2$ such that $G/N$ is of prime order and \eqref{a:1}--\eqref{a:2} are satisfied. Then
\[spd(G) \ge \frac{ f(p,\alpha_1,\alpha_2)}{2 \ |\mathrm{sn}(G)| \ |\mathrm{M}(G)|},\]
where $f(p,\alpha_1,\alpha_2)=\frac{1}{p^2-2p+1} \Big( (\alpha_2-\alpha_1+1)p^{\alpha_1+3}+2p^{\alpha_1+2}-
(\alpha_2-\alpha_1-1)p^{\alpha_1+1}-(\alpha_1+\alpha_2-1)p^2-(\alpha_1+\alpha_2+11)p+(\alpha_1+\alpha_2+5)\Big)$ is a polynomial function depending only on $N$.
\end{lem}

\begin{proof} 


We note that $G=HN$, where $G/N\simeq H$ is of prime order, so that it is meaningful to formulate the conditions in \eqref{a:1} and \eqref{a:2}, requiring that they are satisfied.
From \eqref{lb3},
\begin{equation}\label{solv:1}spd(G) \ge \frac{|\mathrm{sn}(N)| \ |\mathrm{M}(N)| spd(N)+|\mathrm{sn}(G/N)| \ |\mathrm{M}(G/N)| spd(G/N) }{2 \ 
|\mathrm{sn}(G)| \ |\mathrm{M}(G)|}\end{equation}
 $|\mathrm{sn}(G/N)|=|\mathrm{M}(G/N)|=2$, by a counting argument, and $spd(N)=spd(G/N)=1$, by Corollary \ref{c:3}, then 
\begin{equation}\label{solv:2}
= \frac{|\mathrm{sn}(N)| \ |\mathrm{M}(N)|}{2 \ |\mathrm{sn}(G)| \ |\mathrm{M}(G)|} +\frac{4}{2 \ |\mathrm{sn}(G)| \ |\mathrm{M}(G)|}
= \frac{1}{2 \ |\mathrm{sn}(G)| \ |\mathrm{M}(G)|} \left( |\mathrm{sn}(N)| \ |\mathrm{M}(N)|  + 4\right)\end{equation} 
\cite[Theorem 3.3]{mtbis} implies $|\mathrm{sn}(N)|=|\mathrm{L}(N)|=\frac{1}{(p-1)^2}[(\alpha_2-\alpha_1+1)p^{\alpha_1+2}-(\alpha_2-\alpha_1-1)p^{\alpha_1+1}-(\alpha_1+\alpha_2+3)p+(\alpha_1+\alpha_2+1)]$, and, as noted above, $|\mathrm{M}(N)|=\frac{p^2-1}{p-1}=p+1$, hence
\begin{equation}\label{solv:3}
= \frac{1}{2 \ |\mathrm{sn}(G)| \ |\mathrm{M}(G)|}  \ \cdot \  \Big(\frac{p+1}{(p-1)^2} \Big((\alpha_2-\alpha_1+1)p^{\alpha_1+2} -(\alpha_2-\alpha_1-1)p^{\alpha_1+1}\end{equation}
\[-(\alpha_1+\alpha_2+3)p+(\alpha_1+\alpha_2+1)\Big)+ 4\Big) \]
in order to write better the above expression we introduce the coefficients
\[C_1=\alpha_2-\alpha_1+1; \ C_2=\alpha_2-\alpha_1-1; \ C_3=\alpha_1+\alpha_2+3; C_4=\alpha_1+\alpha_2+1\]
and then we get
\[
= \frac{1}{2 \ |\mathrm{sn}(G)| \ |\mathrm{M}(G)|} \Big(\frac{1}{(p-1)^2} \  (C_1p^{\alpha_1+3}+(C_1-C_2)p^{\alpha_1+2}-
C_2p^{\alpha_1+1}\]\[+(4-C_3)p^2-(8+C_3)p+(4+C_4)\Big).\] 
Developping  the computations in the brackets, we get the  polynomial $f(p,\alpha_1,\alpha_2)$.
\end{proof}

Lemma \ref{l:1} may be adapted to $sd(G)$ in the following way.

\begin{lem}\label{l:2} Let $N=\mathbb{Z}_{p^{\alpha_1}} \times \mathbb{Z}_{p^{\alpha_2}}$ be a non--trivial normal subgroup of $G$ with $0\le \alpha_1+ \alpha_2$ and $1 \le \alpha_1 \le \alpha_2$ such that $G/N$ is of prime order. Then
\[sd(G) \ge \frac{g(p,\alpha_1,\alpha_2)}{2 \ |\mathrm{L}(G)|^2},\] where $g(p,\alpha_1,\alpha_2)= \frac{1}{(p-1)^4}\Big((\alpha_2-\alpha_1+1)p^{\alpha_1+2}-(\alpha_2-\alpha_1-1)p^{\alpha_1+1}-(\alpha_1+\alpha_2+3)p+(\alpha_1+\alpha_2+1)\Big)^2+ 4$ is a polynomial function depending only on $N$.
\end{lem}

\begin{proof} We note that $G=HN$, where $G/N\simeq H$ is of prime order.
 \eqref{a:1} is in this case $\mathrm{L}(G/N)^2 \sim \mathrm{L}(H)^2 \subseteq \mathrm{L}(G)^2$ and is always satisfied. Analogously, \eqref{a:2} becomes $\mathrm{L}(N)^2 \subseteq \mathrm{L}(G)^2$ and is satisfied, too. Then \eqref{lb3} becomes
\begin{equation}\label{solv:1bis}sd(G) \ge \frac{|\mathrm{L}(N)|^2 sd(N)+|\mathrm{L}(G/N)|^2 spd(G/N) }{2 \ |\mathrm{L}(G)|^2}\end{equation}
and, from the assumptions, $|\mathrm{L}(G/N)|=2$, $sd(G/N)=sd(N)=1$ but again \cite[Theorem 3.3]{mtbis} implies $|\mathrm{L}(N)|^2=\frac{1}{(p-1)^4}\Big((\alpha_2-\alpha_1+1)p^{\alpha_1+2}-(\alpha_2-\alpha_1-1)p^{\alpha_1+1}-(\alpha_1+\alpha_2+3)p+(\alpha_1+\alpha_2+1)\Big)^2$. Therefore
\begin{equation}\label{solv:2bis}
=\frac{1}{2 \ |\mathrm{L}(G)|^2} \Big( \frac{1}{(p-1)^4}\Big((\alpha_2-\alpha_1+1)p^{\alpha_1+2}-(\alpha_2-\alpha_1-1)p^{\alpha_1+1}\end{equation} \[-(\alpha_1+\alpha_2+3)p+(\alpha_1+\alpha_2+1)\Big)^2+ 4\Big),\]
where it appears clear what is the polynomial function $g(p,\alpha_1,\alpha_2)$, which we are looking for.
\end{proof}



As usual $Fit(G)$ denotes the  \textit{Fitting subgroup} of $G$. 

\begin{thm}\label{t:1} Let $G$ be a solvable group in which $C=C_G(Fit(G))=\mathbb{Z}_{p^{\alpha_1}} \times \mathbb{Z}_{p^{\alpha_2}}$, for  $0\le \alpha_1+ \alpha_2$, $1\le \alpha_1 \le \alpha_2$, $p$ a prime  and $|G:C|$ a  prime. 
\begin{itemize} \item[(i)]If \eqref{a:1}--\eqref{a:2} are satisfied, then
$spd(G) \ge \frac{ f(p,\alpha_1,\alpha_2)}{2 \ |\mathrm{sn}(G)| \ |\mathrm{M}(G)|},$
where $f(p,\alpha_1,\alpha_2)$ is a polynomial function depending only on $C$.
\item[(ii)] $sd(G) \ge \frac{g(p,\alpha_1,\alpha_2)}{2 \ |\mathrm{L}(G)|^2},$ where $g(p,\alpha_1,\alpha_2)$ is a polynomial function depending only on $C$. 
\end{itemize}
\end{thm}

\begin{proof} Since $G$ is solvable, it is well--known that $C$ is an abelian normal subgroup of $G$. Then our position is correct in assuming  $C=\mathbb{Z}_{p^{\alpha_1}} \times \mathbb{Z}_{p^{\alpha_2}}$, with $0\le \alpha_1+ \alpha_2$, $1\le \alpha_1 \le \alpha_2$, $p$ prime and  $G/C$ is of prime order. Now (i) is an application of Lemma \ref{l:1} and (ii)  of Lemma \ref{l:2}.
\end{proof}

The lower bound in Lemma \ref{l:2} for $sd(G)$ is more precise than the following bound, which was the first to be presented in literature.

\begin{cor}[See \cite{mt}, Corollary 2.6]\label{2.6}
A group $G$ possessing a normal abelian subgroup of prime index has $|\mathrm{L}(G)|^2 \ sd(G)\ge |\mathrm{L}(N)|^2+2|\mathrm{L}(N)| +1$.
\end{cor}

A different  restriction is obtained when we multiply up \eqref{lb1}--\eqref{lb2}.

\begin{prop}\label{p:1}  Let  $N$ be a normal subgroup of a group $G=NH$ satisfying \eqref{a:1} and \eqref{a:2}. Then \[spd(G)\ge \frac{1}{|\mathrm{sn}(G)| \ |\mathrm{M}(G)|}    \sqrt{{\underset{(Z,T)\in \mathrm{sn}(H) \times \mathrm{M}(H)}{\underset{(X,Y)\in \mathrm{sn}(N) \times \mathrm{M}(N)} \sum}} \chi(X,Y) \ \chi(Z,T)}.\]
\end{prop}

\begin{proof}  From \eqref{lb1}--\eqref{lb2} and the Cauchy inequality for numerical series,
\[ |\mathrm{sn}(G)|^2 \ |\mathrm{M}(G)|^2 \ spd(G)^2\ge  \   \sum_{(X,Y)\in \mathrm{sn}(N) \times \mathrm{M}(N)} \chi(X,Y) \ \cdot \ \sum_{(Z,T) \in \mathrm{sn}(H) \times \mathrm{M}(H)} \chi(Z,T)\]
\begin{equation}\label{lb4} \ge   {\underset{(Z,T)\in \mathrm{sn}(H) \times \mathrm{M}(H)}{\underset{(X,Y)\in \mathrm{sn}(N) \times \mathrm{M}(N)} \sum}} \chi(X,Y) \ \chi(Z,T).
\end{equation}
All are positive quantities and then, extracting the square root, the result follows.
\end{proof}

The next result answers in a certain sense to \cite[Problem 4.1]{mt}.

\begin{cor}\label{c:p1}  Let  $N$ be a normal subgroup of a group $G=NH$. Then 
\[   sd(G)\ge \frac{1}{|\mathrm{L}(G)|^2}    \sqrt{{\underset{(Z,T)\in \mathrm{L}(H)^2}{\underset{(X,Y)\in \mathrm{L}(N)^2} \sum}} \chi(X,Y) \ \chi(Z,T)}.\]
\end{cor}
\begin{proof}Firstly, we note that the corresponding versions of \eqref{a:1} and \eqref{a:2} for $sd(G)$ are always satisfied. Then we argue as in Proposition \ref{p:1}.
\end{proof}

\section{Applications and final considerations}
The symmetric group on 3 elements $S_3=\mathbb{Z}_2 \ltimes \mathbb{Z}_3=\langle a,b \ | \ a^3=b^2=1, b^{-1}ab=a^{-1}\rangle$ has $sd(S_3)=\frac{5}{6}$ (see \cite[p.2510]{mt}), is metabelian and satisfies the description in Theorem \ref{t:1}, since it is an example of a \textit{primitive group of affine type} (see \cite{dm}). This group was the origin of our investigation. In fact, a primitive group $P$ of affine type is a semidirect product with normal factor $Fit(P)$. Furthermore, $Fit(P)$ turns out to be elementary abelian  and  $C_P(Fit(P))=Fit(P)$. This means that Theorem \ref{t:1} gives a good description for the subgroup commutativity degree and for the subgroup S--commutativity degree of such groups. While \cite{dn,elr,r1,nr,r2,r3} show that we may classify a group, when restrictions on $d(G)$ are given, the problem is still open for $sd(G)$ and $spd(G)$. We illustrate  one case, involving $sd(G)$. This is to justify the interest in Section 2 in the new bounds. 

\begin{cor}
A metabelian group $G$ with $|G'|$ and $|G/G'|$ of prime orders is cyclic, whenever the  bound in Corollary \ref{2.6} is achieved with  $sd(G)=\frac{5}{6}$.
\end{cor}

\begin{proof} We begin from 
$|\mathrm{L}(G)|^2 \ \frac{5}{6} = |\mathrm{L}(N)|^2+2|\mathrm{L}(N)| +1$, which becomes
$|\mathrm{L}(G)|^2 \ \frac{5}{6} = 4+4 +1=9$,  then $2\le|\mathrm{L}(G)|=\sqrt{\frac{56}{5}}=\sqrt{11.2}<4$. This implies either $|\mathrm{L}(G)|=2$ or $|\mathrm{L}(G)|=3$. In the first case, $G$ is cyclic of prime order. In the second case, $G$ is lattice isomorphic to $C_{p^2}$ for a suitable prime $p$. In both cases $G$ is cyclic.
\end{proof}

The control of $|\mathrm{L}(G)|$ was the main ingredient in the previous proof.
Unfortunately, formulas for the growth of $\mathrm{L}(G)$ are  hard to find and  \cite{sh}  helps our investigations. The \textit{M\"obius number} of  $\mathrm{L}(G)$ is a number which allows us to control the size of $|\mathrm{L}(G)|$. In case of a symmetric group $S_n$, it is denoted by $\mu(1,S_n)$ and was conjectured to be $(-1)^{n-1} \ (|\mathrm{Aut}(S_n)|/2)$ for all $n>1$ (see \cite[p.1]{sh}). For $n\leq 11$, this was proved by H. Pahlings. Recent progresses are summarized below.

\begin{thm}[See \cite{sh}, Theorems 1.6, 1.8, 1.10]\label{t:sh}\
\begin{itemize}
\item[(i)] Let $p$ be a prime. Then $\mu(1,S_p)=(-1)^{p-1} \ \frac{p!}{2}$.
\item[(ii)]Let $n=2p$ and $p$ be an odd prime. Then
\[\mu(1,S_n)=\left\{\begin{array}{lcl} -n!,&\,\,& \mathrm{if} \ n-1  \ \mathrm{is \ prime \ and} \ p \equiv 3\mod 4,\\
\frac{n!}{2},&\,\,& \mathrm{if} \ n=22,\\
-\frac{n!}{2},&\,\,& \mathrm{otherwise.}
\end{array}\right.\]
\item[(iii)]  Let $n=2^a$ for an integer $a\geq 1$. Then $\mu(1,S_n)=-\frac{n!}{2}$.
\end{itemize}
\end{thm}

 Let $\mu(1,G)\in\{\mu(1,S_p),\mu(1,S_n)\}$, being $\mu(1,S_p)$ and $\mu(1,S_n)$ the values in Theorem \ref{t:sh} under the given restrictions on $n$ and $p$. In a certain sense, the following result specifies our considerations on $S_3$, when we have an arbitrary primitive group of affine type for which the subgroup lattice is growing as  $S_n$.


\begin{cor}\label{bound}
Under the assumptions of Theorem \ref{t:1}, let $G$ be a solvable group such that $|\mathrm{L}(G)|=\mu(1,G)$. Then
$sd(G) \ge \frac{g(p,\alpha_1,\alpha_2)}{2 \ \mu(1,G)^2},$ where $g(p,\alpha_2,\alpha_2)$ is a polynomial function depending only on $C_G(Fit(G))$.
\end{cor}

\section*{Acknowledgement}
The second author thanks  D.E. Otera  and the Universit\'e Paris-Sud 11 for the  hospitality  during the month
of May 2010, in which the significant part of the present project has been done. We are also grateful to some colleagues, 
who  detected two fundamental problems in the original version of the manuscript, allowing us to enlarge our perspectives of study in the present version.

\end{document}